\documentclass[10pt]{amsart}
\usepackage{amsmath,amsfonts,amssymb,amscd}

\usepackage{graphics, graphicx}
\usepackage{epsfig}
\usepackage{multirow}
\usepackage{array}
\usepackage{enumitem}
\usepackage{changes}

\newtheorem{theorem}{Theorem}[section]
\newtheorem{proposition}[theorem]{Proposition}
\newtheorem{corollary}[theorem]{Corollary}
\newtheorem{lemma}[theorem]{Lemma}
\newtheorem{remark}[theorem]{Remark}
\newtheorem{definition}[theorem]{Definition}
\newtheorem{example}[theorem]{Example}
\newtheorem{notation}[theorem]{Notations}

\def\NZQ{\Bbb}
\def\NN{{\NZQ N}}

\def\ZZ{{\NZQ Z}}

\def\CC{{\NZQ C}}

\def\PP{{\NZQ P}}

\def\SS{{\NZQ S}}
\def\TT{{\NZQ T}}

\def\ml{\mathcal{C}}
\def\ml1{\mathcal{C}^1}
\def\mlb1{\mathcal{C}_{b}^{1}}

\def\frk{\frak}

\def\Phi{{\frk n}}

\def\Codim{{\rm codim}}
\def\Card{{\rm Card}}
\def\dim{{\rm dim}}

\def\A{{\mathcal A}}
\def\B{{\mathcal B}}
\def\H{{\mathcal H}}
\def\L{{\mathcal L}}

\newcommand{\C}[0]{\mathbb{C}}

\newcommand{\Zt}[0]{\mathcal Z}

\newcommand{\Cal}[1]{\mathcal{#1}}

\newcommand{\bd}[1]{\textbf{#1}}

\renewcommand{\P}{\bd P}

\newcommand{\codim}[0]{\mbox{codim }}

\begin{document}

\title{Strata of discriminantal arrangements}

\author{Anatoly Libgober} 
\address{Department of Mathematics, 
University of Illinois, Chicago, IL 60607} 
\email{libgober@uic.edu}

\author{Simona Settepanella}
\address{Department of Mathematics, 
University of Hokkaido, Sapporo, Japan.}
\email{s.settepanella@math.sci.hokudai.ac.jp}



\begin{abstract} We give an explicit description of the 
multiplicities of codimension two strata 
of discriminantal arrangements introduced by Manin and Schechtman. As applications, 
we discuss the connection of these results with properties of Gale transform and 
we calculate the fundamental groups of the complements
to discriminantal arrangements. 
\end{abstract}

\maketitle

\centerline {\it In memory of Brieskorn}
\noindent

\section{Introduction}

In 1989, Manin and Schechtman (\cite{MS}) 
introduced a family of arrangements of hyperplanes 
generalizing classical braid arrangements, which they called 
the {\it discriminantal arrangements} (p.209 \cite{MS}). 
Such  an arrangement $\B(n,k,\A^0), n,k \in {\bf N}$ 
for $k \ge 2$ depends on a choice 
$\A^0=\{H^0_1,...,H^0_n\}$ of a collection of hyperplanes in 
the general position in $\CC^k$, i.e., such that $\dim\bigcap_{i \in K,
 \Card~K=k} H_i^0=0$. It consists of 
parallel translates of $H_1^{t_1},...,H_n^{t_n}, (t_1,...,t_n) \in \CC^n$ 
which fail to form a general position arrangement in $\CC^k$.
$\B(n,k,\A^0)$ can be viewed as a  
generalization of the pure braid 
group arrangement (\cite{OT}) with which  
$\B(n,1)=\B(n,1,\A^0)$ coincides. 
These arrangements have several beautiful 
relations with diverse problems in the areas such as combinatorics (see \cite{MS} and also 
\cite{Crapo}, which is an earlier appearance of discriminantal
arrangmements), the Zamolodchikov equation with its relation to 
 higher category theory (see Kapranov-Voevodsky \cite{Kap}),
and the vanishing of cohomology of bundles on toric varieties
(\cite{Perling}). 

The purpose of this note 
is to study the dependence of $\B(n,k,\A^0)$ on the data
$\A^0=\{H^0_1,...,H^0_n\}$ . Paper \cite{MS} concerns the arrangements
$\B(n,k,\A^0)$ for which the intersection lattice is constant when $\A^0$
varies within a Zariski open set $\Zt$ in the 
space of general position arrangements. 
However \cite{MS} does not describe the set $\Zt$ explicitly. 
It was shown in \cite{Falk} 
that, contrary to what was frequently 
stated (see for instance \cite{Or}, sect. 8, \cite{OT} or \cite{RL}), the combinatorial type of $\B(n,k,\A^0)$ indeed depends on 
the arrangement $\A^0$ . This was done by providing an example of a
discriminantal arrangement with a combinatorial 
type distinct from the one which occurs when $\A^0$
varies within the Zariski open set $\Zt$. 
Few years later, in \cite{Athan}, Athanasiadis provided a full
description of combinatorics of $\B(n,k,\A^0)$ when $\A^0$ belongs to
$\Zt$. In particular, in this case, codimension $2$ strata of
$\B(n,k,\A^0)$ only have a multiplicity equal to $2$ or $k+2$ . Following \cite{Athan}, we call arrangements $\A^0$ in $\Zt$ {\it very generic}.

Our main result describes a  necessary and sufficient {\it geometric} condition  
on arrangement $\A^0$ assuring that $\B(n,k,\A^0)$ admits codimension $2$ strata of multiplicity $3$.

 This condition is given in terms of a notion of \textit{dependency} 
for the arrangement $\A_{\infty}$ in $\PP^{k-1}$ of hyperplanes
$H_{\infty,1},...H_{\infty,n}$ which are the intersections of projective closures of $H_1^0,...,H_n^0 \in \A^0$ with the hyperplane at infinity. 
Consider three groups of $s \in \ZZ_{\ge 1}$ hyperplanes in $\PP^{2s-2}$
such that together these $3s$ hyperplanes are in general position in
$\PP^{2s-2}$. If the three subspaces corresponding to this split in groups, each being the intersection of hyperplanes in
each group, span a hyperplane in $\PP^{2s-2}$, we say that the arrangement of
$3s$ hyperplanes in $\PP^{2s-2}$ is \textit{dependent} (Definition \ref{depend} in Section  \ref{mainsection}). This dependence condition defines a proper Zariski closed subset of the space of arrangements of $3s$
hyperplanes in $\PP^{2s-2}$ in general position. Our main result
(Theorem \ref{proofmain}) shows that $\B(n,k,\A^0), k>1$ admits a codimenion two stratum of
multiplicity $3$ if and only if $\A_{\infty}$ is an arrangement
in $\PP^{k-1}$ admitting a restriction\footnote{Here restriction is the standard restriction of arrangements to subspaces as defined in \cite{OT} (see also equation (\ref{eq:depen}) in this paper).} which is a dependent arrangement . 

Subsequently, in Section \ref{corol}, we interpret this result 
in terms of the Gale transform. The relation between discriminantal 
arrangements and the Gale transform can be seen, at least implicitly,  already 
in paper \cite{Falk}. From this view point our result asserts 
an equivalence of certain types of collinearity: the dependency of $\A_{\infty}$ is equivalent to presence of dependencies in the Gale
transform which in turn is equivalent to the presence of strata of multiplicity 3
in an arrangement $\B(n,k,\A^0)$. 
We shall give a direct verification 
of such equivalences using the interpretation of Gale transform of six-tuples of
point in $\PP^2$ in terms of  del Pezzo surfaces given in \cite{DO}.  
More precisely, an arrangement $\B(6,3,\A^0)$ depends on arrangement at infinity $\A_{\infty}$, which in this case is a six-tuple of lines in $\PP^2$,  or
equivalently, a six-tuple $(\A_{\infty})^*$
of points in the dual plane. A general position arrangement $\A_{\infty}$ is dependent  if and only if the del Pezzo surface,
which is the blow up of $\PP^2$ at six-tuple $(\A_{\infty})^*$, admits an Eckardt point (cf. subsection \ref{cubics}).
On the other hand, the interpretation of $\B(6,3,\A^0)$ via Gale transform, 
described in subsection \ref{galeassociated},
shows that presence in $\B(6,3,\A^0)$ of codimension
two strata of multiplicity 3 is equivalent to the following: {\it the Gale transform of}
$(\A_{\infty})^*$ is a six-tuple $G(\A_{\infty})^*$ such that blow up of $\PP^2$ 
at $G(\A_{\infty})^*$  is a del Pezzo surface admitting an Eckardt 
point. 
Hence the main result in the Theorem
\ref{proofmain}, in the case of discriminantal arrangments $\B(6,3,\A^0)$, becomes an invariance of 
existence of Eckardt points in the Gale transform. 
We show that this can  be verified directly (see subsection \ref{cubics}). 

Finally we supplement R.Lawrence's presentation (\cite{RL}) by giving a presentation of the fundamental
group in the case of non very generic arrangements (i.e. for 
which $\A^0 \notin \Zt$). 
In fact, we give calculations 
yielding the braid monodromy and hence a presentation of 
the fundamental group of the complement to 
a discriminantal arrangement in all cases.

Notice that in the case $k=1$, the complement to the discriminantal arrangement
$\B(n,1)$ coincides with the configuration space of ordered $n$-tuples
of points in $\CC$. A natural generalization of this configuration
space to the case $k \ge 2$ is the space of arrangements of
hyperplanes in $\CC^k$ in a general position. This is a Zariski open in
the product of $n$ copies of spaces of affine hyperplanes in $\CC^k$. 
The fundamental group of this space is another natural candidate for a
generalization of the pure braid group. Our result shows the
difficulty with a calculation of this fundamental group: natural maps
between spaces $\B(n,k,\A^0)$ for various $n,k$, which in the case $k=1$
lead to the presentation of the pure braid group fail to be locally
trivial fibrations and hence fails to produce an exact sequence of fundamental
groups. For example, intersections of projective closures
of arrangements in $\CC^k$ with the hyperplane at
infinity, yields a map from the space of general position arrangements in
$\CC^{k}$ to the space of general position arrangements in $\PP^{k-1}$. Our result
shows that this map is a locally trivial fibration only over the space of
very general position arrangements in $\PP^{k-1}$. 
Calculation of  the fundamental groups
of spaces of general position arrangements of lines will be addressed elsewhere.

The content of the paper is the following. In Section \ref{prelim},
we introduce several notions used later and recall 
definitions from \cite{MS}. Section \ref{mainsection}
contains one of the main results of this paper, Theorem \ref{proofmain},
describing the codimension 2 strata of discriminantal arrangements having
multiplicity 3 and showing an absence of codimension 2 strata
having a multiplicity different from $2, 3$ and $k+2$.
 The Section \ref{corol} contains the interpretation of the results in
 Section \ref{mainsection} in terms of the Gale transform. The last Section
describes the braid mondromy and fundamental groups 
of the complements to discriminantal arrangements.

Finally, the authors\footnote{\textbf{Acknowledgment}: The first named author was supported by a grant from Simons Foundation and the second maned one by JSPS Kakenhi [Grant Number 26610001 to S.S.].}
 wants to thank I.Dolgachev for a useful
comment on the material in Section \ref{corol}, B. Guerville-Balle for useful comments
and Max Planck Institute 
and University of Hokkaido for hospitality during visits
to these institutions where much of the work on this project was done.

\section{Preliminaries}\label{prelim}

\subsection{Discriminantal arrangements.}\label{discrimsubsection}

Let $H^0_i, i=1,...,n$, be a general position arrangement in $\CC^k, k<n$, i.e., 
a collection of hyperplanes such that $\dim\bigcap_{\substack{i \in K,\\
 \Card K=k}}H_i^0=0$. 
The space of parallel translates $\SS(H_1^0,...,H_n^0)$ (or simply $\SS$ when the
dependence on $H_i^0$ is clear or not essential)
is the space of $n$-tuples
$H_1,...,H_n$ such that either $H_i \cap H_i^0=\emptyset$ or 
$H_i=H_i^0$ for any $i=1,...,n$.
One can identify $\SS$ with an $n$-dimensional affine space $\CC^n$ in
such a way that $(H^0_1,...,H^0_n)$ corresponds to the origin.

We will use the compactification of the arrangement
$(H_1^0,....,H_n^0)$ obtained by viewing the ambient space $\CC^k$  
as $\PP^k\setminus H_{\infty}$ endowed with a collection of hyperplanes
$\bar H^0_i$ which are projective closures of affine hyperplanes
$H^0_i$. The condition of genericity is equivalent to $\bigcup_i \bar H^0_i$ 
being a normal crossing divisor in $\PP^k$.
The space $\SS$ can be identified with product 
$\L_1 \times \ldots \times \L_n$ 
where $\L_i \simeq \CC$ is the pencil of hyperplanes spanned by $H_{\infty}$ and
$H^0_i$ parametrized by $\PP^1$ with the deleted point. 
The latter corresponds to $H_{\infty}$ 
and the origin to $H^0_i$. In particular, an ordering of
hyperplanes in $\A$ determines the coordinate system in $\SS$.

For a general position arrangement $\A$ in $\CC^k$ 
formed by hyperplanes $H_i, i=1,...,n$,
{\it the trace at infinity} (denoted by $\A_{\infty}$)  is the arrangement 
formed by hyperplanes 
$H_{\infty,i}=\bar H^0_i\cap H_{\infty}$.

An arrangement $\A$ (or its trace $\A_{\infty}$)
determines the space of parallel translates $\SS(H_1^0,...,H_n^0)$ (as a subspace
in the space of $n$-tuples of hyperplanes in  $\PP^k$).

For a general position arrangement $\A_{\infty}$, we consider the closed subset 
of $\SS(H_1^0,...,H_n^0)$ formed by those collections 
which fail to form a general position arrangement. 
This subset is a union of hyperplanes with each 
hyperplane corresponding to a subset 
$K=\{i_1,...,i_{k+1}\}\subset \{1,...,n\}$ and consisting 
of $n$-tuples of translates of hyperplanes $H_1^0,...,H_n^0$ 
in which translates of $H_{i_1}^0,...,H_{i_{k+1}}^0$ 
fail to form a general position arrangement 
(equations are given by (\ref{hypeq}) below). 
Such a hyperplane will be denoted $D_K$.
The corresponding arrangement will 
be denoted $\B(n,k,\A)$ and called 
{\it the discriminantal arrangement corresponding to} $\A$.

The cardinality of 
$\B(n,k,\A)$ is equal to $n \choose k+1$. 
Each hyperplane $D_K$ contains the $k$-dimensional subspace 
$\TT$ of 
$\SS(H_1^0,...,H_n^0)$ formed by $n$-tuples of hyperplanes 
containing a fixed point in $\CC^k$.
Clearly, 
the essential rank, i.e. the dimension of
  the ambient space minus the dimension of intersection of the
  hyperlanes of the arrangement (cf. \cite{Stanley}), in the case of $\B(n,k,\A)$ is
$n-k$ and the arrangement induced by the
arrangement of hyperplanes $D_K$ in the quotient 
of $\SS(H_1^0,...,H_n^0)$ by $\TT$ is essential.
It is called {\it the essential part} of the discriminantal arrangement.

\subsection{Hyperplanes in $\B(n,k,\A)$}\label{hyperplanessec}

Recall that an arbitrary arrangement $\A$ of 
hyperplanes $W_1,...,W_N \subset \CC^k$ 
defines the canonical stratification of $\CC^k$ 
in which strata are defined as follows. Let $L(\A)$ be the intersection poset 
of subspaces in $\CC^k$, each being the intersection of 
a collection of hyperplanes chosen among $W_1,...,W_N$, and for each $P
\in L(\A)$, let $\Sigma_P=\{i \in  \{1,...,N\} \vert P \in W_i \}$ be 
the set of indices of hyperplanes $W_i$ such that $P = \cap_{i \in \Sigma_P} W_i$.
Vice versa,  given a subset $\Sigma \subset \{1,...,N\}$, we denote by
$w_{\Sigma}$ the subspace $w_{\Sigma}=\cap_{i \in \Sigma} W_i$.
The stratum of $P$ is the submanifold of $\CC^k$ defined as follows:
\begin{equation}
  {\mathcal S}_P =P \setminus \bigcup_{\Sigma_P \subset \Sigma}  w_{\Sigma} \quad .
\end{equation} 

If an arrangement $\A=\{W_1,...,W_N\}$ in $\CC^k$ is {\it in the general position} then the
finite subset in $\CC^k$, consisting of 
0-dimensional strata, has cardinality $N \choose k$ 
and its elements are in one to one correspondence with the subsets of 
$\{1,...,N\}$ having cardinality $k$.

The multiplicity of a point $p \in {\mathcal S}_P$ considered as a point on the
subvariety $\bigcup_{i=1,...,N} W_i$ in $\CC^k$ is constant along the stratum. 
We call it {\it the multiplicity} of the stratum ${\mathcal S}_P$.
It is equal to cardinality of the set $\Sigma_P$.

As we noted, 
the hyperplanes of $\B(n,k,\A)$ correspond to subsets of cardinality $k+1$ 
in $\{1,\ldots,n\}$. Their equations can be obtained as follows.
Let $K,  \Card~K =k+1$, be a subset in $\{1,...,n\}$ and let 
\begin{equation}\label{matrixofequations}
\alpha^j_1 y_1+\ldots+\alpha^j_{k}y_k=x_j^0, \quad j \in \{1,...,n\}
\end{equation}
be the equation of hyperplane $H_j^0$ of arrangement 
$\A=\{H_1^{0},\ldots,H_n^{0}\} \in \CC^n\setminus B(n,k,\A)$
in selected coordinates $y_1,...,y_k$ in $\CC^k$. 
The hyperplanes $H_j$, $j \in K$, of an arrangement 
in $\SS$ with equations 
$\alpha^j_1 y_1+\ldots+\alpha^j_{k}y_k=x_j, \quad j \in K$, will have non-empty 
intersection iff
\begin{equation}\label{hypeq}
    det \begin{pmatrix}\alpha_1^1 & ... & \alpha_k^1 & x_1 \cr
   ... & ... & ... & ... \cr
     \alpha_1^{k+1} & ... & \alpha_k^{k+1} & x_{k+1} \cr
\end{pmatrix} = 0 \qquad .
\end{equation}
This provides a linear equation in $x_j,j=1,...,k+1$, for the hyperplane
$D_K$ corresponding to $K$. 

Let  $J$ be a subset 
in $\{1,\ldots,n\}$ of cardinality $a$,  
\begin{equation}\label{codim2subspaces}
D_J=\{(H_1,\ldots,H_n) \in 
\SS \mbox{ such that } \cap_{i \in J} H_i \neq \emptyset \} \quad 
\end{equation}
and
\begin{equation} 
\Cal{P}_{k+1}(J)=\{K \subset J \mbox{ such that } \Card~K =k+1\} \quad .
\end{equation} 
Then 
\begin{equation}\label{eq1:inter}
D_J=\bigcap_{K \in \Cal{P}_{k+1}(J)}D_K
\end{equation}
is intersection of ${a} \choose {k+1}$ hyperplanes. 
In particular $D_J$, $\Card~J \ge k+1 $, is a linear subspace 
and the multiplicity of $\bigcup_{\Card~K =k+1} D_K$ at its 
generic point is ${a}\choose k+1$. Moreover, $\Codim~D_J$ is $a-k$.

\subsection{Projections of discriminantal arrangements}\label{projectionsdisc}
Let $\Xi \subset \{1,...,n\}$ be a subset of the set of indices and let
$\SS(\Xi) \subset \SS$ be the subspace of the space of translates of
hyperplanes 
of a general position arrangement $H^0_1,...,H^0_n$ consisting of translates
of hyperplanes with indices in $\Xi$.  Let us consider the projection $p_{\Xi}: \SS
\rightarrow \SS(\Xi)$ obtained by omitting from a collection  
of translates from $\SS$, the translates of hyperplanes with indices outside of
$\Xi$. The image of a subspace $D_J, J \subset \{1,...,n\}$ is a proper
subspace iff $\Card~J \cap \Xi \ge k+1$ and in fact $p_{\Xi}(D_J)=D_{J
  \cap \Xi}$. In particluar, if $D_J$ is a hyperplane
  i.e. $\Card J=k+1$ then $p_{\Xi}(D_{J})$ is a hyperplane if and only
  if $J \subset \Xi$. 

The maps $p_{\Xi}$ restricted to the complement to
  the discriminantal arrangement $\SS\setminus \B(n,k,\A)$ for $n \ge k+3$ are locally trivial fibrations if and only if
$k=1$.  Due to their local triviality they play a prominent role in the study of braid
arrangements (cf. \cite{Birman}). The failure of local
  triviality for $k \ge 2$ can be seen as follows. Consider, for example, the simplest
  case $k=2$. Let $\A=\{l^0_1,..,l^0_4,l^0_5\}$ be a quintuple of lines in $\CC^2$
  and $\Xi=\{1,2,3,4\}\subset \{1,2,3,4,5\}$. The fiber of $p_{\Xi}:
  \CC^5 \rightarrow \CC^4$ at a generic point $\{l_1,....,l_4\}$ in the complement $\CC^4 \setminus \B(4,2,\A \setminus \{l^0_5\})$ is given by all general position arrangements $\{l_1,....,l_4,l_5\}$ such that $l_5$ does not contain any of the 6 intersection points of $l_i \cap l_j$, $1 \leq i<j \leq 4$, that is $\CC$ with deleted 6
 points. On the other hand, one can select a generic point
 $\{l_1,....,l_4\}$ in the complement $\CC^4 \setminus \B(4,2,\A
 \setminus \{l^0_5\})$ such that one of the diagonals of quadrangle
 formed by lines $l_1,...l_4$ will be parallel to $l^0_5$. Hence the
 fiber of $p_{\Xi}$ at such a point will be $\CC$ with only 5 points deleted. Similar special
configurations are inevitable for all $n\ge k+3, k\ge 2$. This failure
of local trivaility brings serious complication in the study of the
topology of the complement $\SS\setminus \B(n,k,\A)$ (see the last
section for a description of the fundamental groups).

Note, that some recent works (see for example \cite{FSV}) refer
to discriminantal arrangements 
in a more narrow sense than used in this paper i.e. 
as the restriction arrangements to the fibers of $p_{\Xi}$  
given explicitly as  
\begin{equation}
    p_{\{1,...,l\}}^{-1}(t_1,...,t_l)= \{(z_1,...,z_{n-l}) \vert z_i =z_j \ \ {\rm or}  \ \  z_i = t_k,
       \    k=1,...,l, \ i,j=1,...,n-l \} 
\end{equation}

\section{Codimension two strata having multiplicity 3}\label{mainsection}

In this section we describe necessary and sufficient
conditions which should be satisfied by the trace at infinity $\A_{\infty}$ in order  that the
corresponding discriminantal arrangement will have codimenion two
strata 
having multiplicity 3. We shall start with the list of
notations used throughout this section, some already introduced in the
last section.

\begin{notation}\label{notation1} Let's fix the following notations.
\begin{itemize}

\item $\A^0$ is a general position arrangement of $n$ hyperplanes in
$\CC^k$ ( we use $\A^0$ for the fixed arrangement to distinguish it from $\A$ which will denote a general translate of $\A^0$), 
\item for each $K$ subset of $\{1,...,n\}$ of $\Card~K=k+1$,
 $D_{K} \subset \CC^n$ will denote the hyperplane in 
$\B(n,k,\A^0)$ corresponding to the subset $K$. 
\item As in  subsection \ref{discrimsubsection}, hyperplanes 
in the trace at infinity  $\A_{\infty}$  are denoted by $H_{\infty,i}$.
\item Let $s \ge 2$.  $K_i,i=1,2,3$, denote subsets of $\{1,...,n\}$ such that 
$\Card~K_i=2s, \Card~K_i\cap K_j=s$, $i \ne j$, $\bigcap_{i=1}^{i=3} K_i=\emptyset$
(in particular $\Card~\bigcup K_i=3s$). 
\end{itemize}
\end{notation}

\begin{lemma}\label{lem:case3} 
Let $s \ge 2$, $n=3s, k=2s-1$. Let $\A^0$ be a general position arrangement of $n$ 
hyperplanes in $\CC^k$ and let $K_i,i=1,2,3$ be a triple of subsets of
$\{1,...,n\}$ as described in notations \ref{notation1} above. Consider the triple of codimension $s$ 
{\bf subspaces} of the
hyperplane at infinity $H_{\infty}$
defined as follows:  $H_{\infty,i,j}= \cap_{ s \in K_i \cap K_j} H_{\infty,s} \cap H_{\infty}, i \ne j$.
If subspaces $H_{\infty,i,j} \subset H_{\infty}$ span a proper subspace 
in $H_{\infty}$ then $\Codim~\bigcap D_{K_i}=2$. Otherwise this codimension is 
equal to 3.
\end{lemma}

This lemma suggests the following:

\begin{definition}\label{depend} A general position arrangement in $\PP^{2s-2}, s \ge 2$, 
is called {\bf dependent} 
if it is composed of $3s$ hyperplanes
  $W_i$ which can be  
partitioned into 3 groups, each
containing $s$ hyperplanes, such
that 3 subspaces of dimension $s-2$ , each being intersection 
of hyperplanes in one group, span a proper subspace in $\PP^{2s-2}$. 
We call these three $s-2$-dimensional subspaces {\bf dependent}.
\end{definition}

Remark that, with this terminology, the assumption of Lemma
  \ref{lem:case3} is that the trace at infinity of $\A^0$ is a dependent general position arrangement.

If $s=2$ in Lemma \ref{lem:case3}, then $H_{\infty,i,j}$ are points in
the $2$ dimensional space $\PP^2$. The condition that these points span a proper
subspace in $H_{\infty}$, i.e., are collinear, corresponds to
the case of Falk's example of the special discriminantal arrangement in \cite{Falk}. 
We shall illustrate the argument in Lemma \ref{lem:case3} by
a discussion of this particular case since the argument for the proof of this lemma is a
generalization of the argument used in Example \ref{6lines}.
\begin{example}\label{6lines} 
Let us consider the case $n=6$ and $k=3$, that is a general position arrangement
$\Cal{A}^0=\{H^0_1,\ldots,H_6^0\}$ in $\CC^3$.  In Lemma
\ref{lem:case3}, this corresponds to $s=2$ and, after possible
relabelling,  $K_1=(1,2,3,4),K_2=(3,4,5,6),K_3=(1,2,5,6)$.  Then
subspaces $L_{i,j}=\bigcap_{s \in K_i\cap K_j}H_s^0$ are lines
$L^0_{1,3}=H_1^0 \cap H_2^0,L^0_{1,2}=H_3^0 \cap H_4^0,L^0_{2,3}=H_5^0\cap
H_6^0$ with closures $\bar L_{i,j}^0$.   In this case, (i.e., when $\dim H_{\infty}=2$), the assertion of 
Lemma \ref{lem:case3} is that the points $H_{\infty,i,j}=\bar L^0_{i,j}
\cap H_{\infty}$ span a line $l$ in $H_{\infty}$. In other words, the points 
$H_{\infty,i,j}$ are collinear if and only if $\codim_{i=1,2,3} \bigcap D_{K_i}=2$ (see Figure \ref{fig:case1}).\\
\begin{figure}[t] 
\centering
\includegraphics[scale=0.4]{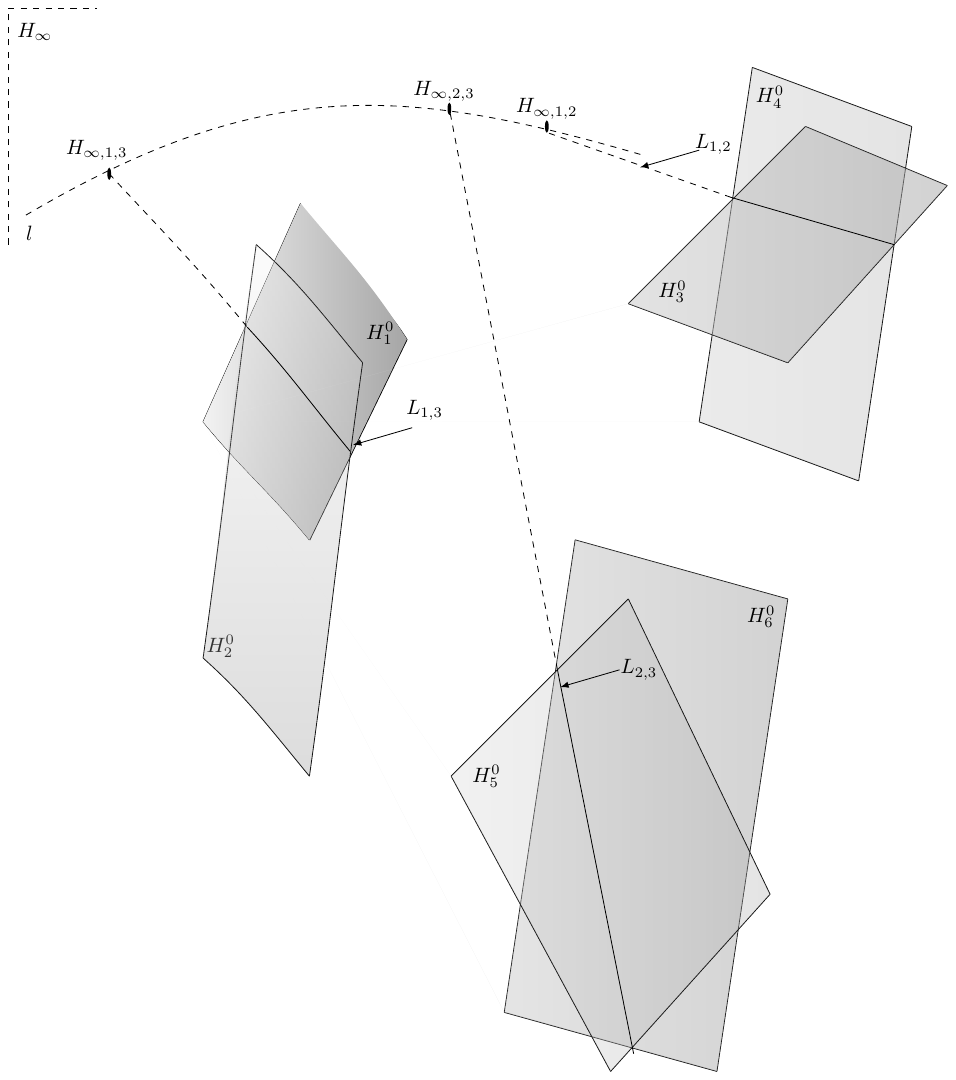}
\caption{\small }\label{fig:case1}
\end{figure}
\indent Indeed, an arrangement $\Cal{A}=\{H_1,\ldots,H_6\}$ of  translates of
planes in $\Cal{A}^0$ is a point in $D_{K_1}\cap D_{K_2}$ iff
pairwise intersections $L_{1,3}\cap L_{1,2}$ and $L_{1,2}\cap L_{2,3}$
in $\CC^3$ 
of lines $L_{1,3}=H_1 \cap H_2, L_{1,2}=H_3 \cap
H_4$ and  $L_{2,3}=H_5 \cap H_6$ are non-empty.
We claim that the collinearity condition implies that 
two pairs of these three lines $L_{i,j}$ are coplanar if and only if all the three are. Indeed, since
$\Cal{A}$ 
consists of translates of planes in $\Cal{A}^0$ the line $L_{i,j}$ has
the same point at infinity $H_{\infty,i,j}$ as does the line
$L^0_{i,j}$. The condition
that  $H_{\infty,i,j}$ span a line $l\in H_{\infty}$ implies that
the closure of any
plane containing two lines $L_{i,j}$ intersects $H_{\infty}$ in $l$.
That is two planes containing respectively the pairs of lines
$L_{1,3},L_{1,2}$ 
and $L_{1,2},L_{2,3}$
are
coincident. This implies that lines $L_{1,3}$ and $L_{2,3}$ have
a non-empty intersection i.e. $\bigcap_{i=1,2,5,6} H_i\ne \emptyset$ and 
hence $\A \in D_{K_3}$.

Vice versa, if the points $H_{\infty,i,j}$ aren't colinear, then it is possible to find configurations in which, for example,  $L_{1,3}$ intersects both $L_{1,2}$ and $L_{2,3}$, but $L_{1,2} \cap L_{2,3} = \emptyset$, i.e. $\Cal{A} \in D_{K_1} \cap D_{K_3}$ and $\Cal{A} \notin D_{K_2}$.
\end{example}

\begin{proof}[\textbf{Proof of Lemma \ref{lem:case3}}] Consider first the case when subspaces 
$H_{\infty,i,j}$ span a proper hyperplane in $H_{\infty}$ 
which we shall denote $\H$.
Note that 
$\dim~H_{\infty,i,j}=s-2$ and,  
as a consequence of $\A^0$ being in the general position, these subspaces do not 
intersect. In particular, the subspace which they span has a dimension greater than
$2s-4$, i.e., either it is a hyperplane or it is all the space $H_{\infty}$.

Let $\A=\{ H_i\}$ be an arrangement in $\CC^k=\CC^{2s-1}$ which belongs to 
$D_{K_1}$ and $D_{K_2}$ (recall that hyperplanes $H_i$ are translates of
hyperplanes $H_i^0 \in \A^0$ ).  Hence $\bigcap_{i \in K_1} H_i \ne \emptyset$ and
$\bigcap_{i \in K_2} H_i \ne \emptyset$. We claim that 
$\bigcap_{i \in K_3} H_i \ne \emptyset$, which would imply 
that $\Codim~\bigcap_{i=1,2,3} D_{K_i}=\Codim~\bigcap_{i=1,2}D_{K_i}=2$.
Let $L_{i,j}=\bigcap_{s \in K_i\cap K_j, }  H_s, (i<j)$.
Note that the codimension of each linear subspace $L_{i,j}$ of $\CC^{2s-1}$ 
is equal to $s$ and $L_{i,j}\cap H_{\infty}=H_{\infty,i,j}$.

Since $\A \in D_{K_1}$, the subspaces $L_{1,2}$ and $L_{1,3}$ 
have a non-empty intersection. Therefore they 
span in $\CC^{2s-1}$ a hyperplane which we denote as $\L_1$. The intersection of $\L_1$ 
with $H_{\infty}$ is the hyperplane $\H$ spanned 
by $H_{\infty,1,2}$ and $H_{\infty,1,3}$.
The hyperplane $\L_1$ is spanned by the intersection 
point $L_{1,2} \cap L_{1,3}$ and the hyperplane $\H$. 

Similarly, since $\A \in D_{K_2}$, both $L_{1,2}$ and $L_{2,3}$ 
have a point in common, they span the hyperplane $\L_2$
spanned by this point and the above hyperplane $\H$ which can be described 
as the plane spanned by $H_{\infty,1,2}$ and $H_{\infty,1,3}$.
Both hyperplanes $\L_1$ and $\L_2$ contain $L_{1,2}$ and $\H$.
Hence they coincide. Therefore $L_{1,3}$ and $L_{2,3}$, 
being both $(s-1)$-dimensional subspaces in $\L_1=\L_2, \dim \L_1=\dim \L_2=2s-2$, must have a point in common 
and hence $\A \in \bigcap_{i=1,2,3} D_{K_i}$.

Now assume that the triple $H_{\infty,i,j}$ spans $H_{\infty}$.
Let   $\A \in D_{K_1}\cap D_{K_2}$ be sufficiently generic in this space.
We show that it does not belong to $D_{K_3}$. Consider 
the family of $s$ codimensional subspaces in $\CC^{2s-1}$ which compactification 
intersects the hyperplane at infinity at $H_{\infty,2,3}$.
The selection of $\A$ determines subspaces $L_{1,2},L_{1,3} \subset \CC^{2s-1}$ which have 
a common point and moreover the subspace $L_{2,3}$ which intersects $L_{1,2}$.
Since triple $H_{\infty,i,j}$ is {\it not} in a hyperplane in
$H_{\infty}$, $\PP^{2s-1}$, compactifying $\CC^{2s-1}$ is spanned by 
$H_{\infty,2,3}$ and the closures of subspaces $L_{1,2},L_{1,3}$.
Hence the generic subspace $L$ of codimension $s$ containing $H_{\infty,2,3}$ and intersecting 
$L_{1,2}$ will have an empty intersection with $L_{1,3}$. The corresponding arrangement 
$\A'$ having $L$ as the subspace $L_{2,3}$ will not belong to $D_{K_3}$ but will be in 
$D_{K_1}\cap D_{K_2}$. This shows that $D_{K_1}\cap D_{K_2} \notin D_{K_3}$.
\end{proof}

Let's briefly recall here the basic notion of the  restriction of an arrangement. For a subset $\A' \subseteq \A$, let us denote by $X_{\A'}=\bigcap_{H
\in \A'} H$ the intersection of its hyperplanes. The arrangement
\begin{equation}\label{eq:depen}
\A^{X_{\A'}}=\{H \cap X_{\A'} \mid H \in \A \setminus \A', H \cap X_{\A'} \neq \emptyset\}
\end{equation}
is called \textbf{a restriction} of $\A$ to  $X_{\A'}$. Restrictions of $\A$ are in one-to-one
correspondence with the  
splits $\A=\A'\bigcup \A''$ of the set of hyperplanes in $\A$ into 
a disjoint union. 
If $\A'=\emptyset$, then the restriction arrangement
coincides with $\A$.

Via the restriction of arrangements, the Lemma \ref{lem:case3} leads to
other examples of discriminantal arrangements having codimension two
strata with multiplicity 3.

\begin{lemma}\label{lem:rest}  
Let $\A^0$ be a general position arrangement of
  $n$ hyperplanes in $\CC^k$ and $\A'$ be a subarrangement of $t$
  hyperplanes in $\A^0$.
Assume that the trace at infinity of the restriction $\A^{X_{\A'}}$ of $\A^0$ to 
$X_{\A'}$ is a dependent arrangement of $3s=n-t$ hyperplanes (in the sense of Def. \ref{depend}).
Then $\B(n,k,\A)$ admits a codimension two stratum of
multiplicity 3.
\end{lemma}

\begin{proof} 
Assume that $\A'=\{H^0_1, ... ,H^0_t\} \subset \A^0=\{H^0_1, ... ,
H^0_n\}$ satisfies the conditions of lemma, i.e., the restriction $\A^{X_{\A'}}$ is
an arrangement of $3s=n-t$ hyperplanes in $X_{\A'} \simeq \CC^{n-t}$
and its trace at infinity $\A_{\infty}^{X_{\A'}}$ is dependent, i.e., the discriminantal 
arrangement $\B(n-t,k-t,\A^{X_{\A'}})$ admits a codimension 2 stratum
having the multiplicity 3. 
The dimension of this stratum is $3s-2$ where $n-t=3s$ and $k-t=2s-1$.  
By Lemma \ref{lem:case3}, there are subsets $K_i,i=1,2,3$, $\Card~K_i=2s=k-t+1$ of $\{t+1,...,n\}$ such that $D_{K_i} \in \B(n-t,k-t,\A^{X_{\A'}})$ satisfy $\Codim~\bigcap_{i=1,2,3}
D_{K_i}=2$. The above $3s-2$ dimensional stratum 
of the discriminantal arrangement of $n-t$ hyperplanes in $\CC^{k-t}$ 
is the transversal intersection of two submanifolds (each being an
open subset of a linear
subspace) of $\CC^n$. One 
is the stratum of discriminantal arrangement
$\B(n,k,\A^0)$  having the dimension  $3s-2+t$ formed by hyperplanes 
 $D_{K_i \cup  \{1,...,t\}}, i=1,2,3$, and another is the intersection of
 $t$ hyperplanes in
 $\SS(H_1^0,...,H^0_n)$ defined by the vanishing of coordinates corresponding
 to $H_1^0,...,H_t^0$.
 Hence the multiplicity of this stratum of $\B(n,k,\A^0)$ equals 3. This
 yields the lemma.
\end{proof}

\begin{corollary}\label{inequality}
 If $k \ge 3$ and $n \ge k+3$, then there exists a
  general position arrangement of $n$ hyperplanes in $\CC^k$
  such that the corresponding discriminantal arrangement admits 
a codimension two stratum of multiplicity 3. 
\end{corollary}

\begin{proof} To apply Lemma \ref{lem:rest}, for a pair $(n,k)$ such
  that there exist integers $t \ge 0, s
  \ge 2$ satisfying 
\begin{equation}\label{dep:rel}
n=3s+t \ \ \ \ \ k=2s-1+t
\end{equation} 
  consider a general position
  arrangement $\A^0$ of $n$ hyperplanes in
  $\CC^k$ such that the restriction of trace $\A_{\infty}$ of $\A^0$ on intersection of its $t$
  hyperplanes is dependent. By Lemma \ref{lem:rest}, the discriminantal
  arrangement corresponding to such $\A^0$ will admit the required stratum. Given $(n,k)\in
\NN^2$, the relation (\ref{dep:rel}) has a unique solution
$s=n-k-1,t=3k-2n+3$ which satisfies $s\ge 2, t\ge 0$ iff
\begin{equation}\label{formulainequality}
k+3 \le n \le {3 \over 2}(k+1), \  \ \ \ k\ge 3 \quad .
\end{equation}
Note that given an arrangement $\B(n,k,\A)$ admitting the codimension 2
strata of multiplicity 3, an extension of $\A$ to the arrangement of $N
\ge n$ hyperplanes by adding sufficiently generic hyperplanes yields
an arrangement $\A'$ such that  $\B(n,k,\A)$ is the
intersection of $\B(N,k,\A')$ and the coordinate subspace.
It follows that $\B(N,k,\A')$ admits strata of codimension 2
and the multiplicity 3 as well.
On the other hand, for $n=k+2$, $\B(k+2,k,\A)$ has only one
stratum of multiplicity $k+2$ i.e., the inequality $n \ne k+3$ is
sharp.  
\end{proof}

The following example illustrates the above two lemmas.

\begin{example} Let $\A_{\infty}$  be a general position arrangement of $8$ hyperplanes $H_{\infty,i}$ in
  $\PP^4$ and $\A_{\infty}^X$ its restriction to the plane $X=H_{\infty,7}\cap H_{\infty,8}$. The restricted arrangement
$\A_{\infty}^X$ is an arrangement in general position since
  $\A_{\infty}$ is in general position.

 Assume that the double points 
 $H_{\infty,1}\cap H_{\infty,2}\cap H_{\infty,7}\cap   H_{\infty,8}$,  
$H_{\infty,3}\cap H_{\infty,4}\cap H_{\infty,7}\cap   H_{\infty,8}$,
$H_{\infty,5}\cap H_{\infty,6}\cap H_{\infty,7}\cap   H_{\infty,8}$
are co-linear.

Consider the hyperplanes $$D_{1,2,3,4,7,8},D_{3,4,5,6,7,8},D_{1,2,5,6,7,8}$$
in a discriminantal arrangement $\B(8,5,\A)$
corresponding to such $\A_{\infty}$ and the hyperplanes
$D_{1,2,3,4}',D_{3,4,5,6}',D_{1,2,5,6}'$ in the discriminantal arrangement 
in 3-space $H_7\cap H_8$ for a generic choice of hyperplanes $H_7,H_8$ 
intersecting the hyperplane at infinity at $H_{\infty,7},H_{\infty,8}$ respectively.
Then the arrangement $\A$ of $8$ hyperplanes in $\CC^5$ including $H_7,H_8$ 
 has a common point if and only if the arrangement of $6$ planes in
 3-space 
$H_7\cap H_8$ has a common point. Hence 
\begin{equation}\dim D_{1,2,3,4,7,8}\cap D_{3,4,5,6,7,8} \cap D_{1,2,5,6,7,8}=
2+\dim D_{1,2,3,4}'\cap D_{3,4,5,6}'\cap D_{1,2,5,6}'=6
\end{equation}
(the last equality uses the Example \ref{6lines}).
Hence the discriminantal arrangement $\B(8,5,\A)$ has a codimension two stratum of multiplicity 3.\\
This case illustrates the case considered in Theorem \ref{proofmain} (2)
  below, corresponding to the dependent restriction arrangement of $\A_{\infty}$ given by hyperplanes $H_{\infty,i} \cap H_{\infty,7}\cap   H_{\infty,8}$, $i=1,\ldots,6$ and $s=2$.

\end{example}

The next Lemma will be useful in the proof  below showing the absence of codimension 2 strata 
having  the multiplicity 4.

\noindent 
\begin{lemma}\label{quadruples} For $s \ge 2$, there is no quadruple of subspaces $V_i \subset
  \PP^{3s-2}, i=1,2,3,4$ having dimension $2s-2$ such that intersections
$P_{i,j}=V_i \cap
  V_j,i \ne j$
satisfy 

 a) each $P_{i,j}$ has 
  dimension $s-2$ 

b) any pair $P_{i,j},P_{i,k}, i \ne j \ne k \ne i$ spans a hyperplane in
$V_i$, and 

c) all three,  $P_{i,j},P_{i,k},P_{i,l}$, belong to a hyperplane in $V_i$.

\end{lemma}

\begin{proof} We shall start with the case $s=2$. Assume that
 a configuration as in Lemma \ref{quadruples} does exist and consider a quadruple
  of planes $V_i, i=1,...,4$ in
  $\PP^4$ such that 

a) any two intersect at a single point, 

b) all 6 points $P_{i,j}=V_i \cap V_j$, $i \neq j$,  obtained in this way are distinct, and 

c) all three points, $P_{i,j},P_{i,k},P_{i,l}$, are colinear, i.e., span a
line $L_i$.

For a fixed $k$, the triple of points $P_{i,j}$,$i, j  \ne k$, outside
of $V_k$, determines the triple of
  lines $L_i \subset V_i, i \ne k$ spanned by points
  $P_{i,j},P_{i,l}$, $i, j , l \ne k$. 
These lines $L_i, i \ne k$ by their definition are pairwise concurrent
($L_i\cap L_j=P_{i,j}$) and hence belong to a plane $H$.
By assumption c), for each $i \ne k$, the three points $P_{k,i}=V_k
\cap V_i$ are points on lines $L_i$ distinct from $P_{i,j},P_{i,l}$. 
Hence $H$ and $V_k$ have 3 distinct non-colinear points in common
and therefore  $H=V_k$, but this contradicts 
$\dim~V_k \cap V_i=0$. 

Now consider the case $s>2$. Similarly to the above, $s-2$ dimensional 
subspaces $P_{i,j}=V_i \cap
V_j$ of $\PP^{3s-2}$ determine the subspaces $L_i\subset
V_i, i \ne k$ (for a fixed $k$) each being spanned by pairs $P_{i,j}, P_{i,l},i,j,l \ne
k$ which are outside of $V_k$. Each $L_{i}$ is a hyperplane in $V_i$ 
(i.e. $\dim~L_i=2s-3$).
Moreover, the dimension of the subspace $H$ of $\PP^{3s-2}$ spanned by 
$L_{i,j},L_{i,l}, i,j,l \ne k$, is $3s-4$. The subspace $H$ can be described as the
subspace 
of $\PP^{3s-2}$ spanned by triple of subspaces $P_{i,j}, i,j \ne k$.
Now by our assumption c), $V_k$ contains an $s-2$ dimensional 
subspace of $L_{i,j}, i,j \ne k$, i.e., $P_{i,k}$. The subspace hence is also a
subspace of $H$. This implies that
$V_k \subset H$. The dimension of intersection $L_{i,j}$ and $V_k$, 
which are both subspaces of $H$, is equal to $(2s-3)+(2s-2)-(3s-4)=s-1$ 
and hence $\dim~V_i\cap V_k=s-1$. This is a contradiction. 
\end{proof}

Now we are ready for the main result of this section. It describes
the codimension 2 strata of discriminantal arrangements having
the multiplicity 3 and shows an absence of codimension 2 strata
having the  multiplicity 4 (with obvious exceptions). 

\begin{theorem}\label{proofmain}

Let $\A_{\infty}$ be a general position arrangement of  hyperplanes 
in $\PP^{k-1}$ which is the trace at infinity of a general position arrangement $\A^0$ in $\C^k$. 

1. The arrangement $\B(n,k,\A^0)$ has 
$n \choose k+2$ codimension 2 strata of multiplicity $k+2$.

2. There is a one-to-one correspondence between

a) the dependent restrictions of subarrangements of $\A_{\infty}$,   and  

b) triples of hyperplanes in $\B(n,k,\A^0)$ for which the codimension 
of their intersection is equal to 2.

3. There are no codimension 2 strata having the multiplicity 4 unless
$k=2$. All codimension 2 strata of  $\B(n,k,\A^0)$ not 
mentioned  in part 1, have a multiplicity which is either
2 or 3 (the latter corresponding to triples of hyperplanes in b).

4. The codimension 2 strata of $\B(n,2,\A^0)$ is independent of $\A^0$.
\end{theorem}

\begin{proof}
The statement (1) 
follows immediately from the discussion after 
(\ref{eq1:inter}) in section \ref{hyperplanessec}.
If $J \subset \{1,...,n\}$ is a subset of cardinality $k+2$, 
then $D_J$ is a codimension 2 subspace in $\CC^n$
and belongs to $k+2$ hyperplanes $D_K, K \subset J$.

Next we shall determine the conditions on three different sets of
$k+1$ indices  
under which 
$\Codim~D_{K_i}\cap D_{K_j} \cap D_{K_l}=2$.

Consider first the case when sets $K_i,K_j,K_l$, each having
the cardinality $k+1$, are such that 
for one of them,  say $K_i$, one has 
$K_i\setminus (K_i\cap (K_j \cup K_l))\ne \emptyset$, 
i.e., one of the set in this triple is not in the union of other two.
If $r \in K_i\setminus (K_i\cap (K_j \cup
K_l))$, 
then the hyperplanes in an arrangement $\A \in D_{K_j}\cap D_{K_l}$ with indices different
from the indices in $K_j\cup K_l$ 
can be chosen as arbitrary parallel translates of hyperplanes in
$\A^0$,
while $H_r \in \A', \A' \in D_{K_i}$ is fixed by condition $\A'$ being
in $D_{K_i}$ and the selection of hyperplanes with indices different from
$r$ but in $K_i$. 
Hence $D_{K_i}\cap D_{K_j}\cap D_{K_l} \ne D_{K_j}\cap D_{K_l} $, i.e., 
$\Codim~D_{K_i}\cap D_{K_j}\cap D_{K_l}=3$.

Now let us consider the alternative to the case considered in the
previous paragraph. Hence we have a triple $K_i,K_j,K_l$ such that
\begin{equation}\label{reducedK_i}
K_i=(K_i \cap
K_j) \bigcup (K_i\cap K_l)
\end{equation} 
for any permutation of $(i,j,k)$.  Condition (\ref{reducedK_i}) for $k=2$ implies that either  $\Card (K_i \cap
K_j)=2$ or $\Card (K_i\cap K_l)=2$ that is either $D_{K_i}\cap D_{K_j}=D_{K_i \cup
K_j}$ or respectively $D_{K_i} \cap D_{K_l}=D_{K_i \cup
K_l}$. Since this imples that $\Card(K_i\cup K_j)=4$ (resp. $\Card
(K_i\cup K_l)=4$), we obtain that $D_{K_i}\cap D_{K_j}$
(resp. $D_{K_i}\cap D_{K_l}$) is a codimension $2$ subspace of multiplicity $4=k+2$ and  part (4) follows.

Let $L_{\alpha,\beta}=(K_{\alpha} \cap K_{\beta})\setminus 
\bigcap_{s=i,j,k} K_{s},
t=\Card~\bigcap_{\alpha=i,j,k} K_{\alpha}, l_{\alpha,\beta}=\Card~L_{\alpha,\beta}$.
Then (\ref{reducedK_i}) implies that $K_{\beta}\setminus 
\bigcap_{\alpha=i,j,k} K_{\alpha}
=L_{\alpha,\beta} \bigcup L_{\beta,\gamma}$ and since $\Card~K_i=k+1$
we have 
\begin{equation}\label{sizeofintersections}
l_{\alpha,\beta}+l_{\beta,\gamma}+t=k+1, \ \ \ \alpha \ne \beta \ne
\gamma
\end{equation}
Using these relations for allowable permutations of subscripts, yields:
\begin{equation}
l_{\alpha,\beta}={{k+1-t} \over 2}    \ \ \ \alpha \ne \beta,
\alpha,\beta \in \{i,j,k\}
\end{equation}

For a triple of subsets $K_i,K_j,K_l, \Card~K_i\cap K_j \cap K_l=t$ and a fixed 
arrangement $\A$,
consider the map of the spaces of translates:
$$\SS(H^0_1,...,H^0_n) \rightarrow 
\CC^{t}=\SS(...,H^0_{r},...) \ \ \ r \in K_i\cap K_j \cap K_l$$
which assigns to a collection of $n$  
parallel translates {$H^{t_1}_1,...,H^{t_n}_n$} of $H^0_1,...,H^0_n$ in $\CC^k$, 
the intersections of the hyperplanes with indices outside of 
{$ K_i \cap K_j \cap K_l$} with the linear subspace which is 
the intersection of $t$ hyperplanes with indices in
 {$ K_i \cap K_j \cap K_l$}. 

This map has as its fiber over the set of translates
$H_{\beta}^{t_{\beta}}$, the space 
$${\SS(...,H^0_{\alpha} \cap  
(\bigcap_{\beta \in K_i \cap K_j\cap K_l}
H^{t_{\beta}}_{\beta}),...)},  \ \ \ \ \alpha \in [1,...,n]\setminus K_i\cap
K_j \cap K_l$$ 
of translates in the $(k-t)$-dimenional
  space $\bigcap_{\beta \in K_i \cap K_j\cap K_l}
H^{t_{\beta}}_{\beta}$.
If  $s$ is the dimension of the  family of arrangements 
which {is} the intersection of hyperplanes $D_{K_\alpha}, \alpha=i,j,l$ 
then the dimension of the family of restrictions of arrangements to $\CC^{k-t}$ is $s-t$. 
Hence 
\begin{equation}\label{intersection}
\Codim~D_{K_i}\cap D_{K_j}\cap D_{K_l}=
\Codim~D_{K_i\setminus \bigcap K_{\alpha}} \cap D_{K_j \setminus \bigcap K_{\alpha}}
\cap D_{K_l \setminus \bigcap K_{\alpha}} \ \ \ \ \alpha=i,j,l 
\end{equation}
where the intersection on the right is taken 
in the space of parallel translates in 
$\bigcap_{j \in \bigcap K_i} H^0_j$.

Clearly $t<k$, and in the case when $t=k-1$, we have $l_{i,j}=1$, i.e., 
$\Card~\bigcup K_i=k+2$ and we are in the case (1), i.e., the codimension 
2 stratum has the multiplicity $k+2$. If $t=0$, then we have 
the case considered in Lemma \ref{lem:case3} and 
we also see from this lemma that the intersection of $D_{K_i}, i=1,2,3$ has 
a codimension two stratum
if and only if the assumptions of the theorem are fulfilled. 
The rest of the part  (2) of the theorem follows from Lemma
\ref{lem:rest}  applied to the restriction on {$\bigcap_{\alpha \in
  K_i \cap K_j\cap K_l} H^0_{\alpha}$}
and the relation (\ref{intersection}) (with $s=l_{\alpha,\beta}$).

Now consider the existence of a codimension 2 strata of multiplicity 4.
Suppose that such stratum exists and $K_i,i=1,...,4$ are the
corresponding subsets of $\{1,...,n\}$. By the quadruples analog
of restriction (\ref{intersection}), it is enough to consider the case
{$\bigcap_{i=1,...,4} K_i=\emptyset$}. Let $l_{i,j,m}=\Card~K_i \cap K_j
\cap K_m$. Then for any $i$, $\Card~K_i\cup \bigcap_{j \ne i} K_j =\Card~
\bigcup K_i$, i.e., $l_{i,j,m}+k+1$ is independent of $(i,j,m)$.
Hence one infers 
from (\ref{sizeofintersections}) the relation $l_{i,j,m}={{k+1} \over 3}$.

Note that $\Codim~\bigcap_{i=1,...,4} D_{K_i}=2$ if and only if 
 $\Codim~D_{K_{i_1}}\cap D_{K_{i_2}}\cap D_{K_{i_3}}=2$ for all 4 triple $1 \le i_j
 \le 4$ of distinct integers. Applying part (2) of the theorem to each
 triple $i_1,i_2,i_3$,  one infers the existence of a quadruple of
 subspaces as in Lemma \ref{quadruples}. Hence this lemma 
implies part (3).

\end{proof}

\begin{corollary} If  a discriminantal arrangement $\B(n,k,\A)$
satisfies  $n>{3 \over
    2}(k+1)$ and admits a codimension 2 stratum of multiplicity 3, then 
there exists a proper subarrangement $\A' \subset \A$ such that 
$\B(n,k,\A')$ admits a codimension 2 stratum of multiplicity 3.
\end{corollary} 

\begin{proof} It follows immediately from the above theorem and 
  inequality (\ref{formulainequality}). 
\end{proof}

\subsection{Numerology of singularities of generic plane sections}

 Theorem \ref{proofmain} contains a complete description of combinatorics
of codimension 2 strata of discriminantal arrangements. Indeed, the possible multiplicities of codimension two strata are ${k+2} \choose {k+1}$, 3 and 2.
The number of points of multiplicity 3 is the number of triples of strata satisfying condition \textit{2a)}. 
It is an interesting problem to determine the number of triple points
which $\B(n,k,\A)$ can have. It is clear from Theorem
\ref{proofmain}
that this number can be arbitrary large when $n \rightarrow \infty$,
though even the 
precise asymptotic is not clear.

\section{The Gale transform and codimension two strata}\label{corol}

\subsection{The Gale transform and associated sets.}\label{galeassociated}

In this subsection we shall discuss interpretation of discriminantal
arrangements using the  Gale transform.
Recall the following: 
\begin{definition} Let $V$ be a vector space over $\CC$, $\dim V=k$, $l_i\in V^*,i=1,...,n$, be $n$ vectors in the dual of
  the  vector space $V$ and let 
   \begin{equation}
     0 \rightarrow V \buildrel  L \over \rightarrow \CC^n \rightarrow W \rightarrow 0
\end{equation}  
be the exact sequence in which $L(v)=(l_1(v),...,l_n(v))$. The Gale transform of collection
$l_i$ is the collection $m_i \in W, i=1,...,n$, of images of the vectors $e_i$ of
the standard basis in $\CC^n$.  
  \end{definition}

The following is suggested by an argument in \cite{Falk}
(see also \cite{Perling} and \cite{CLS}).

\begin{proposition}\label{galeprop}
Let $\A$ be a central arrangement of $\Card~\A=n$ in a
$k$-dimensional vector space $V$ such that the corresponding
arrangement in $\PP^{k-1}$ is in the general position. Let $l_i$ be
the elements in $V^*$ corresponding to the hyperplanes in $\A$.  The essential part of the 
discriminantal arrangement  consists of hyperplanes in $W$ spanned by
collections of $n-k-1$ vectors of the Gale transform of vectors $l_i \in V^*$. 
\end{proposition}

\begin{proof}  Let $f_1,...,f_k$ be  a basis in $V$,
  $x_j, j=1,...,k$, be the coordinates in this basis, and let
  $l_i=\sum a_j^ix_j, j=1,...,k,i=1,...,n$, be the equations of the
 hyperplanes of $\A$. Denote by  $A=\{a^i_j\}$ the corresponding matrix.
Translates of hyperplanes $l_{i_1}=0,...,l_{i_{k+1}}=0$, by
$c_{i_1},...,c_{i_{k+1}}$ respectively, have a non-empty 
intersection if and only if the system of equations $\sum
a_j^{i_s}x_j=c_{i_s}, s={i_1},...,{i_{k+1}}$, has a solution. This takes place
if and only if  the 
projection $\pi_{i_1,...,i_{k+1}}(c)$ of the point $c=(c_1,\ldots,c_n)\in \CC^n$ on the subspace of $\CC^n$ spanned 
by the vectors  $e_{i_1},...,e_{i_{k+1}}$ belongs to the image of projection
$\pi_{{i_1},...,{i_{k+1}}}(L(V))$ of $L(V)$, which is equivalent to 
\begin{equation} \label{inclusion}
c \in
H_{\pi_{{i_1},...,i_{k+1}}}\simeq Span(V,Ker_{\pi_{{i_1},...,i_{k+1}}})
\end{equation} 
(here by an abuse of notation, we identified $V$ with its image $L(V)$ in $\CC^n$).
The hyperplanes $H_{\pi_{{i_1},...,i_{k+1}}}$
form the 
discriminantal arrangement
in $\CC^n$ and the relation $V \subset
H_{\pi_{{i_1},...,i_{k+1}}}$ shows that 
the essential part of discriminantal arrangement 
is its restriction to $W=\CC^n/V$.
The inclusion (\ref{inclusion}) is equivalent  to $c \in Ker _{\pi_{{i_1},...,i_{k+1}}} mod V$. 
The image  $Ker _{\pi_{{i_1},...,i_{k+1}}} \in \CC^n/V=W$ is
spanned by the images of the Gale transform of $l_i, i=1,...,n$, and is the
hyperplane in the essential part of discriminantal arrangement.
\end{proof}

Next recall the classical notion of associated sets (cf. \cite{DO},
Ch.III):
\begin{definition} Let $V$,$W$ be vector spaces such that $\dim V=k, \dim
  W=n-k$. Let $f_1,...,f_k$ and $g_1,...,g_{n-k}$ be the bases of $V,W$ respectively. 
The set of vectors $l_1,...,l_n$ in  $V$ and $m_1,...,m_n$ in $W$ are called 
associated if the matrices $X$ and $Y$ of coordinates of
$l_i,i=1,...,n$, and $m_j,j=1,...,n$, satisfy:
\begin{equation}
                        X \cdot \Lambda \cdot ^tY=0
\end{equation}
\noindent where $\Lambda$ is a diagonal matrix. 
 \end{definition}
The sets in $V$ and $W$ are associated if and only if one is the Gale 
transform of another (see the discussion in \cite{DO} p.33
where the association is discussed in projective setting, for example).

\subsection{Discriminantal arrangements of planes in $\CC^3$ and the Gale transform}\label{cubics}

One can ask for the meaning of a codimension 2 strata with multiplicity three 
in discriminantal arrangements described 
in the Theorem \ref{proofmain} in terms of the Gale transform.
In the case, $n=6,k=3$, one has a geometric interpretation
(see \cite{DO} for geometric interpretations for some other values
$n,k$) of the Gale transform which allows one to show the following: 
\begin{proposition}\label{gale}
 The existence of a partition of 6-tuples of points in
  $\PP^2$ 
into 3 pairs, each pair defining a line, and such that these lines are concurrent lines, is an invariant of
the Gale transform.
\end{proposition}

\begin{remark}
After replacing hyperplanes by the
points of a projective dual space, this proposition is equivalent
 to the case $n=6,k=3$ of the main
theorem (cf. also Example \ref{6lines}). This equivalence follows from
the classical description of the Gale
transform recalled in the proof below.

The more general case, considered in
the Lemma \ref{lem:case3}, can be interpreted as
the following property of the Gale transform $(\PP^{2s-2})^{\times 3s} 
\rightarrow (\PP^s)^{\times 3s}$.

The condition that there is a partition of
$3s$ points in $\PP^s$ into 3 groups of $s$ points, each set spanning
a hyperplane in $\PP^s$ and that, moreover, 
such that these hyperplanes belong to a pencil, is equivalent to the 
condition that the Gale transform of this set of points in $\PP^{2s-2}$ admits a partition into 3
groups of cardinality $s$ such that the triple of $s-1$ dimensional
subspaces, each spanned by an $s$-tuple in $\PP^{2s-2}$, have a non-empty intersection.

This restatement follows immediately  from the dualizaton of hyperplanes of
the general position arrangement. 
Indeed, $3s$ hyperplanes of the general position arrangement in $\CC^{2s-1}$ considered in
Lemma \ref{lem:case3} define $3s$ hyperplanes 
$\PP^{2s-2}$ or equivalently $3s$
points in the dual projective space. The assumption of the dependency of
$3s$ hyperplanes in $\PP^{2s-2}$, after the dualization is
equivalent to requiring that 3 $s-1$-dimensional subspaces $\eta_1^{s-1},\eta_2^{s-1},\eta_3^{s-1}
\subset \PP^{2s-2}$ each spanned by one of 3 subsets of cardinality $s$
(i.e. subsets $K_i\cap K_j$ in notations of definition \ref{depend}) have non-empty intersection.
Since $n-k-1=3s-(2s-1)-1=s$, by Proposition \ref{galeprop}, the hyperplanes of
the essential part of the discriminant arrangment are
spanned by $s$ -subsets of the set of $3s$ points in 
$\PP^s$. Lemma \ref{lem:case3} states that the dependency condition is equivalent to
the existence of the triple of hyperplanes in the discriminantal arrangement
belonging to a pencil of hyperplanes which gives our claim. 

It would be interesting to have a geometric
description of the Gale transform allowing one to show this directly for $s>2$.
\end{remark}

\begin{proof}  We shall use the projective setting which, in this case, relates 
6-tuples of points in $\PP^2$ to another 6-tuples in another
copy of $\PP^2$. Recall that the smooth cubic surfaces in $\PP^3$ 
(i.e., the del Pezzo surfaces of degree $3$) 
can be viewed as  blow ups of a 6-tuples of points in 
$\PP^2$ and the classes of projective equivalence of 
6-tuples of points in $\PP^2$ correspond to isomorphism classes of 
cubic surfaces. The 6-tuple of points in $\PP^2$ is obtained
by contracting 6 of $27$ lines  having pairwise empty intersections.
In terms of the blow up of 6 points in $\PP^2$,  each of $27$ lines is one
of the following: 

1. 6 exceptional curves of the blow up;

2. proper preimages of 15 lines defined by pairs of points;

3. proper preimages of 
6 quadrics determined by a 5 points subset of the blown up 6-tuple.

6-tuples  of lines as above on  a cubic surface $V$ correspond to the following homology classes
in $H^2(V,\ZZ)$:
\begin{equation}
         h_i,i=1,...,6, \ \ \ \ \    (h_i,h_j)=-\delta^i_j \quad .
\end{equation}
Given such 6-tuple $h_i$, one has a unique additional 6-tuple $h_i'$
characterized by the following:
together with $h_i$ the collection $h_i'$ form a double six, i.e., the
following relations are satisfied:
\begin{equation}
   h_ih_j=h_i'h_j'=-\delta^i_j, \ \ \ \ h_ih_j'=1-\delta^i_j \quad .
\end{equation}
Using the description of 27 lines above in terms of lines and quadrics on $\PP^2$ corresponding to the lines
on a del Pezzo surface, the second
component $h_i'$ of a double six, in which 
the first component $h_i$ is  
formed by the 6-tuple of exceptional curves, can be described as
follows. The 6-tuple $h_i'$
 consists 
of the proper preimages of quadrics labeled in the way 
which assigns to (the class of) exceptional curve $h_P$ contracted to a point $P \in \PP^2$ 
(the class of)  the quadric $h_P'$ passing through points of the 6-tuple of points in $\PP^2$ 
distinct  from $P$. 

Now the existence of partition of 6-tuples as in Proposition \ref{gale} is
equivalent to existence of Eckardt point (i.e., a point common to a triple 
of lines on cubic surface) not involving the exceptional curves
and to show Proposition \ref{gale} one needs 
to show that such Eckardt point exists also for the second component 
of a double six. But each line containing a pair of points $P,P'$ on the plane 
$\P$ obtained 
by contraction of a 6-tuples of disjoint exceptional curves on del
Pezzo surface will be passing 
through a pair of 6 points on the plane $\P'$ obtained by contracting 
proper preimages of 6 quadrics on $\P$. 
Indeed, such a line through $P \in \P$ will intersect the proper 
preimage on the blow up of $\P$ 
of the quadric not containing $P$ at exactly one point (corresponding to the 
intersection point with this quadric distinct from the blown up
point).
Since the blow up of 6 points and contracting proper preimages of 
6 quadrics determined by these 6 points is an isomorphism on the 
complement to quadrics  which contains the concurrency point 
of the triple of lines, the claim follows.
\end{proof}

\section{Fundamental groups of the complements to discriminantal
  arrangements }

\subsection{Nilpotent completion of the fundamental group}

In this section we shall describe the nilpotent completion of 
$\pi_1(\B(n,k,\A))$ in the case when $\A$ is not very generic and the 
 corresponding discriminantal arrangement admits a codimension two strata 
of multiplicity 3. This is a direct consequence of \cite{kohno}
Prop.2.2 (see also \cite{MS}).

\begin{proposition}  Completion of the group ring
  $\CC[\pi_1(\CC^n\setminus \B(n,k,\A))]$ with respect to the 
powers of the augmentation ideals is the quotient of the algebra of
non-commutative 
power series $\CC<<X_J>>, J\in$ $\mathcal{P}_{k+1}(\{1,\ldots ,n \})$ 
by the two-sided ideal generated by relations

(i) $[X_J,\sum_I X_I]$  for a pair of subsets $J \in \mathcal{P}_{k+1}(K),$
 with summation over $I \in \mathcal{P}_{k+1}(K), K \subset \{1,...,n\},\Card~K=k+2$.

(ii) $[X_J,X_I+X_J+X_K]$ where $I,J,K$ are subscripts  corresponding to
triples 
of hyperplanes in the discriminantal arrangement, such that
corresponding hyperplanes in $\A_{\infty}$ satisfy dependency condition 
of Theorem \ref{proofmain} (2a).

(iii) $[X_J,X_K]$ for any pair of sets with $\Card~J,K \ge k+3$ and
such that there does not exist subset $I$ such that triple $I,J,K$ satisfies
the conditions of Theorem \ref{proofmain} (2a).
\end{proposition}

\subsection{Braid monodromy of discriminantal arrangements and 
{\bf $\pi_1(\SS\setminus \B(n,k,\A))$}}\label{braidmonodromy}
 We shall describe the fundamental 
group of the complement to a discriminantal arrangement.
In fact, we shall obtain the braid monodromy of the 
generic plane section of discriminantal arrangement, which 
by the classical van Kampen procedure yields the presentation 
of the fundamental group.

We describe the braid monodromy of the generic section 
of $\B(n,k,\A)$ in terms of a collection 
of orderings of hyperplanes of $\B(n,k,\A)$ constructed in 
terms of equations (\ref{matrixofequations}) of 
arrangement $\A$ of hyperplanes in the general position $H^0_j,j=1,...,n$ as follows. 
The generic plane section of $\B(n,k,\A)$
can be described as subset of $\CC^2$ with coordinates $(s,t)$
depending on a choice of generic $a^n,b^n,c^n$ 
(specifying the plane section)  
consisting of points $(s,t)$
such that the rank of the $(k+ 1) \times n$ matrix:
\begin{equation}\label{sectionmatrix}
\begin{pmatrix}\alpha_1^1 & ... & \alpha_k^1 & a^1t+b^1s+c^1 \cr
   ... & ... & ... & ... \cr
     \alpha_1^{n} & ... & \alpha_k^{n} & a^nt+b^ns+c^n \cr
\end{pmatrix}
\end{equation}
is maximal. 
This plane is given in $\SS$ by 
\begin{equation}\label{section}
x_i=a^it+b^is+c^i \quad .
\end{equation}
For fixed $\alpha^i_j,a^i,b^i,c^i, i=1,...,n,j=1,...,k$,
and generic $(t,s)$, the rank of this matrix is $k+1$.
For a generic fixed $s$, there is a finite collection 
$t_1(s)<...<t_{n \choose k+1}(s)$ of real numbers such that 
the rank of (\ref{sectionmatrix}) is $k$: each $t_i(s)$ corresponds
to a $k+1$ subset of $\{1,...,n\}$ labeling a hyperplane in $\B(n,k,\A)$.
Moreover there will be finite collection of real numbers  
$s_1<...<s_N,N \ge {n\choose {k+2}}$
such that for these $s$ there will be strictly less than $n \choose {k+1}$
constants $t$ for which the rank of (\ref{sectionmatrix}) 
is less than $k+1$. In fact, these values $s$ correspond to projections on the $s$-coordinate
of multiple points of the arrangement of lines restriction of $\B(n,k,\A)$ to the $(s,t)$ plane. 
In particular, to each $s_i$ corresponds a subset $P_i$ 
in the sequence $1,...,{n \choose {k+1}}$ 
corresponding to the set of $(k+1)$-subsets yielding the same value
$t(s_i)$.
The cardinality of the
subset $P_i$ is either $k+2$, $3$ 
or 2 (according to the multiplicity of the singular point 
corresponding to $s_i$). 

Recall (cf. for example \cite{Moi} or
\cite{crelle}) that the real line $Im(s)=0$ in the
complex $s$-line $\CC_s\simeq \CC$ can be used to define in a
canonical way the generators of the fundamental group
$\pi_1(\CC_s\setminus \bigcup_{i=1}^N s_i)$ of the complement of $N$
points in $\CC$. In details, the generator
corresponding to the point $s_i, i=1,....,N$ is the loop from a base
point $s_0, Im (s_0)=0, s_0 << 0$, 
to the point $s_i$, circumventing each $s_j, j<i$ as a semi-circle into the halfplane
$Im(s)<0$ and returning back to $s_0$ after making the full circle
around $s_i$. The braid
monodromy for such a path is the product of factors corresponding to each
$s_j, j \le i$ , i.e., the half twist $\beta_{P_j}$
corresponding to $P_j$ for $j<i$ and the full
twist $\beta_{P_i}^2$.

\begin{theorem} 
1.The braid monodromy of a generic plane section corresponding to section
(\ref{section}) of $\B(n,k,\A)$ 
is given by 
\begin{equation}\label{braids} 
 \Pi_{1 \le k \le N}\Gamma_i \ \ \ \mbox{where}  \ \ \ 
 \Gamma_i=\beta_{P_1}^{-1} ... \beta_{P_{k-1}}^{-1}  
\beta_{P_k}^2\beta_{P_{k-1}}...\beta_{P_2}\beta_{P_1}  \quad .  
\end{equation}

2.The fundamental group $\pi_1(\CC^n\setminus \B(n,k,\A))$
has the following presentation:
\begin{equation}          
     \Gamma_i(\delta_j)=\delta_j \ \ j=1,...,{n \choose {k+1}}, i=1,...,N \quad .
\end{equation}
\end{theorem}

These statements are the standard applications of the results from
the theory of braid monodromy (cf., among others, \cite{Moi},\cite{crelle}). 
Different presentations can be obtained via Salvetti's presentation or Randell's presentation for complement of hyperplane arrangements (see \cite{salvetti}, \cite{Ran}). For $\A^0$ very generic, this yields a presentation 
equivalent to the one given in \cite{RL}.

\end{document}